\documentclass[11pt]{amsart}

\usepackage{amsthm}
\usepackage{amsmath}
\usepackage{amssymb}
\usepackage{color}
\usepackage{graphics, graphicx}

\newtheorem{thm}{Theorem}[section]
\newtheorem{theorem}[thm]{Theorem}

\newtheorem{corollary}[thm]{Corollary}

\newtheorem{hadamard}[thm]{Hadamard Conjecture}

\newtheorem{lemma}[thm]{Lemma}
\newtheorem{proposition}[thm]{Proposition}

\newtheorem{definition}[thm]{Definition}
\theoremstyle{remark}
\newtheorem{remark}[thm]{Remark}

\newtheorem{ex}[thm]{Example}

\newcommand{\ZZ}{\mathbb Z}
\newcommand{\RR}{\mathbb R}

\begin{document}

\title{Real matrices whose columns have equal modulus
coordinates}
\author[Botelho-Andrade, Casazza, Cheng, Tran, Tremain
 ]{Sara Botelho-Andrade, Peter G. Casazza, Desai Cheng, Tin Tran, and Janet Tremain}
\email{sandrade102087@gmail.com, casazzapeter40@gmail.com}
\email{ 
chengdesai@yahoo.com, tinmizzou@gmail.com}
\email{tremainjc@missouri.edu}


\begin{abstract}$ $
We study $m \times n$ matrices whose columns are of the form \[\{(a_{1j},\ldots, a_{nj}): 
\quad a_{1j} = 
\lambda_j,\ a_{ij} = \pm\lambda_j\ , \  \lambda_j >0 ,\ j=1,2,\ldots,n\}.\]
 We explicitly construct for all $a = (a_1,\ldots, a_{\frac{m(m-
1)}{2}}) \in \mathbb{R}^{\frac{m(m-1)}{2}}$ a matrix of the above form whose 
rows have pairwise dot product equal to $a$. Using Hardamard matrices constructed by Sylvester we classify 
all 
matrices of the above form whose rows have pairwise dot product equal to $a$. 
We also use our results to reformulate the Hadamard conjecture.
\end{abstract}

\maketitle
\section{Introduction}
Hadamard matrices are named after a J. Hadamard 1893 paper \cite{Ban1}. 
\begin{definition}
A {\bf Hadamard matrix} is a matrix of the form $H_m=[\epsilon_{ij}]_{i,j=1}^m$ where
$\epsilon_{ij}=\pm1$ for all $1\le i,j\le m$ and the rows are orthogonal.
\end{definition}

It follows that $H_mH_m^T=mI_m$.
It is known that Hadamard matrices do not exist unless $m=2,4n$ for $n\ge 1$.
Also note that if we multiply any row or column of $H_m$ by $-1$, it is
still a Hadamard matrix.
This gives rise to the Hadamard Conjecture:

\begin{hadamard}
For every $4n$ a Hadamard matrix exists in $\RR^{4n}$.
\end{hadamard}

The first examples of Hadamard matrices were given by J. J. Sylvester in 1867 \cite{Ban1}.
\begin{equation*}
 H_2=\begin{bmatrix}
1&1\\
1&-1
\end{bmatrix}
\end{equation*}

\begin{equation*}
H_4=\begin{bmatrix}
1&1&1&1\\
1&-1&1&-1\\
1&1&-1&-1\\
1&-1&-1&1
\end{bmatrix}
\end{equation*}
and

\begin{equation*}
H_{2^m}= \begin{bmatrix}
H_{2^{m-1}}&H_{2^{m-1}}\\
H_{2^{m-1}}&-H_{2^{m-1}}
\end{bmatrix}
\end{equation*}

When building a block matrix, we typically use the standard notation where we leave off the brackets of the block, resulting in the standard form of a matrix. However, it is also possible to consider block matrices from a different perspective.

\begin{definition}
If $A$ is an $m\times n$ matrix and $B$ is a $p\times q$ matrix, the {\bf Kronecker product}
$A\otimes B$ is the $pm\times qn$ block matrix:
\[ A\otimes B=\begin{bmatrix}
a_{11}B &\cdots &a_{1n}B\\
\vdots&\ddots&\vdots\\
a_{m1}B&\cdots & a_{mn}B
\end{bmatrix} \]
\end{definition}

It follows that $H_2\otimes H_2=H_4$, and in general, $H_{2^m}=H_2 \otimes H_{2^{m-1}}$,
leading to the following proposition.

\begin{proposition}

For all $m \in \ZZ^{+}$, $H_{2^m} = H_{2}^{\otimes m}$ is a 
Hadamard matrix.
\end{proposition}

%The Walsh matrices were defined by Joseph Walsh in 1923. They are square matrices
%of dimensions $2^n$ where n is a natural number, the entries are $\pm1$ and the
%rows (and hence the columns) are orthogonal.
%The Walsh matrices are a special case of the Hadamard matrices. 
%For example, 
%\[ W=\begin{bmatrix}
%-1&1&1&1\\
%1&-1&1&1\\
%1&1&-1&1\\
%1&1&1&-1
%\end{bmatrix}\]
%is a Walsh matrix, but we will only work with $W_{2^n}=H_{2^n}$.

Recall that a partial Hadamard matrix is a matrix whose entries are of the form $\pm 1$, and its rows are pairwise orthogonal.

Motivated by our quest to find Hadamard matrices, we extend the notion of 
(partial) Hadamard matrices to (non-square) matrices of the form
\[A_{mn}=\begin{bmatrix}
a_{11}&\cdots &a_{1n}\\
\vdots&\ddots &\vdots\\
a_{m1}&\cdots &a_{mn}\\
\end{bmatrix}
\]
\begin{center}
$|a_{ij}| = |a_{i'j}|=\lambda_{j} \quad \mbox{ for all } \quad 1 \leq i, i' \leq m$ and \\
$a_{1j} = \lambda_{j}   \qquad  \mbox{ for all } \quad 1\le j \le n$.
\end{center}

We will study matrices of this form and find relationships between them and Hadamard matrices. These relationships give insight into the combinatorial 
structure of such matrices. We then show that finding a Hadamard matrix of order $m$ is equivalent 
to finding 
a lattice point in $\mathbb{R}^{2^{m-1}}$ with $m$ ones and $2^{m-1} - m$ 
zeros. 

In the next section, we introduce the {\bf Column Pairwise Product} of vectors and apply this to the columns of the truth table. We show that the rows of the resulting matrix are a subset of the rows of a Hadamard matrix. 
To further explore this relationship, we define the concept of {\bf $m$-Hadamardesque} matrices, and develop their properties. We also demonstrate that the classification of Hadamard matrices can be achieved through the use of $m$-Hadamardesque matrices.

\section{Main Results}

In this section for every $m \in \mathbb{Z}^{+}$ we consider all vectors of the 
form $(1, \pm 1, \ldots ,\pm 1 ) \in \mathbb{R}^{m}$. We see that there are $2^{m-1}$ 
vectors of this form and for each one we calculate its contribution to the dot 
product of the rows. In Definition \ref{def1}, for each column we associate a coordinate in 
$\mathbb{R}^{2^{m-1}}$ which 
will measure the amount of ``weight" to give to that column vector.

\subsection{Analysis of the Column Pairwise Product}

\begin{definition}
For a vector $v = (v_1, \ldots, v_m) \in \mathbb{R}^{m}$ define its \textbf{Column Pairwise
Product} to be the vector 
\[v' = (v_1v_2, v_1v_3, v_2v_3,v_1v_4,v_2v_4,v_3v_4, \ldots , v_1v_m, \ldots, v_{m-
1}v_m) \in \mathbb{R}^{(m(m-1))/2}.\]
\end{definition}
Given $A=[a_{ij}]_{i,j=1}^m$ the pairwise dot product of its rows is

\begin{align*}R&=(\langle r_1, r_2\rangle, \langle r_1, r_3\rangle, \langle r_2, r_3\rangle,
\langle r_1,r_4\rangle,\langle r_2,r_4\rangle,\langle r_3,r_4\rangle,\ldots,
\langle r_1, r_m\rangle, \ldots, \langle r_{m-1}, r_m\rangle)\\
&=\left (\sum_{i=1}^ma_{1i}a_{2i},\sum_{i=1}^m a_{1i}a_{3i},\sum_{i=1}^ma_{2i}a_{3i},
\sum_{i=1}^ma_{1i}a_{4i},
\ldots \right ). \end{align*}

Also, the matrix of Column Pairwise Products of the columns of $A$ is
\[ C=\begin{bmatrix}
a_{11}a_{21}&a_{12}a_{22}&a_{13}a_{23}&a_{14}a_{24}&\cdots\\
a_{11}a_{31}&a_{12}a_{32}&a_{13}a_{33}&a_{14}a_{34}&\cdots   \\
a_{21}a_{31}&a_{22}a_{32}&a_{23}a_{33}&a_{24}a_{34}&\cdots     \\
a_{11}a_{41}&a_{12}a_{42}&a_{13}a_{43}&a_{14}a_{44}&\cdots   \\
\vdots &\vdots&\vdots&\vdots&\vdots
\end{bmatrix}
\]
Note that the sum of the column vectors of the matrix $C$ is a column vector whose transpose
is the row vector $R$. So we have:

\begin{proposition}\label{prop4}
The pairwise dot product of a matrix's rows is an element of $\RR^{(m(m-1))/2}$ and
 is equal to the transpose of the sum 
of the Column Pairwise Products of its columns.
\end{proposition}

\begin{remark}
The Column Pairwise Product of $v$ and $-v$ are the same so we may assume the first 
row of our matrix has no negative entries.
\end{remark}

Let us now consider the Column Pairwise Product of each of the $2^{m-1}$ possible 
columns of the form $(1, \pm 1, \ldots ,\pm 1 ) \in \RR^{m}$ (i.e., the truth table). So
\[T_1= \begin{bmatrix}
1 \end{bmatrix}\ \ 
T_2= \begin{bmatrix}
1&1\\
1&-1
\end{bmatrix}
\ \ T_3=\begin{bmatrix}
1&1&1&1\\
1&-1&1&-1\\
1&1&-1&-1
\end{bmatrix}
\]
\[T_4 = \begin{bmatrix}
1&1&1&1&1&1&1&1\\
1&-1&1&-1&1&-1&1&-1\\
1&1&-1&-1&1&1&-1&-1\\
1&1&1&1&-1&-1&-1&-1
\end{bmatrix}
\]
Note that the rows of $T_3$ are the $1,2^0+1,2^1+1$ rows of $H_4$ and the rows of $T_4$ are the
$1,2^0+1,2^1+1,2^2+1$ rows of $H_8$. Now we need to formalize this.

\begin{lemma}\label{prop1}
If for dimension $m$ the set of possible columns of the form $(1, \pm 1, \ldots, \pm 1 ) \in \RR^{m}$ are
the columns of the $m\times 2^{m-1}$ matrix
\[T_m= \begin{bmatrix}
c_{1}&\cdots &c_{2^{m-1}}
\end{bmatrix}
\]
then for dimension $m+1$ the possible columns are
\[T_{m+1} = \begin{bmatrix}
c_{1}&\cdots &c_{2^{m-1}}&c_{1}&\cdots &c_{2^{m-1}}\\
1&\cdots &1&-1&\cdots &-1
\end{bmatrix}
\]
Moreover, in this case all the rows are orthogonal.
\end{lemma}

We will list out the Column Pairwise Products as columns
of a matrix called $CT_m$ where the $i^{th}$ column corresponds to the Column 
Pairwise Product of the $i^{th}$ column from $T_m$. So this is a $\frac{m(m-
1)}{2}\times 2^{m-1}$ 
matrix. This convention will help in 
higher dimensions.

\begin{ex}\label{ex-1}
The two possible columns in dimension 2 are given by the columns of the matrix
\[ T_2=\begin{bmatrix}
1&1\\
1&-1
\end{bmatrix}\] 
The matrix where the $i^{th}$ column corresponds to the Column Pairwise Product
of the $i^{th}$ column above is
\[ CT_2=\begin{bmatrix}
1&-1
\end{bmatrix}\]
\end{ex}

%Let us denote the $n \times n$ Walsh matrix by $W_n$. 
Notice $T_2$ from above 
is the Hardamard matrix
$H_2$ and $CT_2$ is the second  
row of $H_2$.

\begin{ex}
For $m=3$,
\[T_3=\begin{bmatrix}
1&1&1&1\\
1&-1&1&-1\\
1&1&-1&-1
\end{bmatrix}\]

The four possible columns in $CT_3$ are
\[CT_3= \begin{bmatrix}
1&-1&1&-1\\
1&1&-1&-1\\
1&-1&-1&1
\end{bmatrix} \]
\end{ex}

Also, given $T_4$ above we have
\[ CT_4=\begin{bmatrix}
1&-1&1&-1&1&-1&1&-1\\
1&1&-1&-1&1&1&-1&-1\\
1&-1&-1&1&1&-1&-1&1\\
1&1&1&1&-1&-1&-1&-1\\
1&-1&1&-1&-1&1&-1&1\\
1&1&-1&-1&-1&-1&1&1
\end{bmatrix}\]

\begin{proposition}\label{prop2}
The rows of $T_m$ are the rows $\{1, 2^0+1, 2^1+1, \ldots, 2^{m-2}+1\}$ of $H_{2^{m-1}}$. 
\end{proposition}
\begin{proof}
We proceed by induction on $m$. Clearly for $m = 2$, this is true as the rows of 
$T_2$ are the first and second 
rows of $H_2$.  Now assume it is true for $m$. Note that 
\begin{equation*}
H_{2^m}= \begin{bmatrix}
H_{2^{m-1}}&H_{2^{m-1}}\\
H_{2^{m-1}}&-H_{2^{m-1}}
\end{bmatrix}
\end{equation*}
so by Lemma \ref{prop1} the first $m$ rows of $T_{m+1}$ are the rows $\{1, 2^0+1, 2^1+1, \ldots, 2^{m-2}+1\}$ of $H_{2^{m}}$. The last row is $(1,\ldots,1,-1,\ldots,-1) \in 
\RR^{2^m}$ where the first $2^{m-1}$ coordinates are $1$ and the last $2^{m-1}$ 
coordinates are $-1$. This corresponds to the $2^{m-1}+1$ row of $H_{2^{m}}$.
\end{proof}

We now compute $CT_m$ by induction as well. We show that its rows are also rows of 
the Hadarmard matrix. We have already computed $CT_2, CT_3, CT_4$ in Example \ref{ex-1}.
\begin{proposition}\label{prop3}
For $m > 2$, let 
\[CT_{m-1} = \begin{bmatrix}
r_1\\
\vdots\\
r_{\frac{(m-1)(m-2)}{2}}
\end{bmatrix}
\]
and
\[T_{m-1} = \begin{bmatrix}
r'_1\\
\vdots\\
r'_{m-1}
\end{bmatrix}
\]
then
\[CT_{m} = \begin{bmatrix}
(r_{1},r_{1})\\
\vdots\\
(r_{\frac{(m-1)(m-2)}{2}},r_{\frac{(m-1)(m-2)}{2}})\\
(r'_1,-r'_1)\\
\vdots\\
(r'_{m-1}, -r'_{m-1})
\end{bmatrix}
\]
where $(v,w)$ denotes the concatenation of $v$ and $w$.
\end{proposition}
\begin{proof}
Looking at the construction of $T_m$ in Lemma \ref{prop1}, we see the first $m-
1$ rows are simply the rows of $T_{m-1}$ doubled. This means the first $m-1$ 
coordinates of the $i^{th}$ and $i + 2^{m-2}$ columns of $T_m$ are the same for 
$1 \leq i \leq 2^{m-2}$. Therefore the first $\frac{(m-1)(m-2)}{2}$ rows of 
$CT_m$ are 
the rows of $CT_{m-1}$ doubled as stated in the above proposition.

To derive the remaining $m-1$ rows of $CT_m$, we can simply note that the last row of $T_m$ consists of a vector containing $2^{m-2}$ ones and $2^{m-2}$ negative ones. Hence if $r'_i$ is the $i^{th}$
row of $T_{m-1}$ then $(r'_i, -r'_i)$ would be the $(\frac{(m-1)(m-2)}{2} + 
i)^{th}$ row of $CT_{m}$.
\end{proof}

From the definitions of $T_m$ and $CT_m$ we have:

\begin{lemma}\label{lem1}
Let $1\le r_i<r_j\le m$. The 
\[ \left ( \frac{(r_j-1)(r_j-2)}{2}+r_i\right )^{th} \mbox{row in }CT_m,\]
corresponds to the coordinate wise product of the $r_i^{th}$ and $r_j^{th}$
row of $T_m$.
\end{lemma}

Now we can compare the rows of $CT_m$ with the rows of $H_{2^{m-1}}$,
$m\ge 2$.

\begin{proposition}
The rows of $CT_m$ are the rows of $H_{2^{m-1}}$, $m 
\geq 2$.
\end{proposition}
\begin{proof}
We proceed by induction on $m$. From Example \ref{ex-1} this is clearly true 
for $m = 2$. This is our base case.

Now assume that the rows of $CT_{m-1}$ are the rows of $H_{2^{m-2}}$. By Proposition \ref{prop2}, Proposition \ref{prop3},  and the construction of $H_{2^{m-1}}$, it follows that the rows of $CT_m$ are the rows of $H_{2^{m-1}}$.
%From Lemma \ref{prop1} we see the rows of $T_{m-1}$ are also rows of $W_{N'}$ 
%where $N' = 2^{m-2}$. For a row $r_i$ in $W_{N'}$, $(r_i, -r_i)$ is also a row 
%in $W_N$ where $N = 2^{m-1}$. This takes care of the last $m-1$ rows of $CPP_m$.
%Assuming the rows of $CT_{m-1}$ are in $W_{N'}$, since for all rows $r_i \in 
%W_{N'}$, $(r_i,r_i) \in W_N$ for $N$ and $N'$ as above, it is clear the first  
%$\frac{(m-1)(m-2)}{2}$ rows of $CPP_m$ are in $W_N$ as well.
\end{proof}

\begin{remark}
    If we permute the columns of $T_m$ then we can also permute the columns of 
    $H_{2^{m-1}}$ in the same manner.
\end{remark}

\begin{definition}
    For each $m \geq 2$, call \textbf{$R_m$} the rows of $CT_m$. Let 
    \textbf{$RC_m$} be the set of 
    rows in $H_{2^{m-1}}$ which are not in $R_m$
\end{definition}

\begin{remark}
    Since the rows of $H_{2^{m-1}}$ form an equal norm orthogonal basis of $\RR^{2^{m-1}}$, so do 
    $R_m \cup RC_m$. The rows in $R_m$ and $RC_m$ are mutually orthogonal. 
\end{remark}

\subsection{Classifying all matrices whose columns have equal modulus coordinates}

In this subsection we study all matrices whose columns are a positive multiple of 
some column in $T_m$ for each $m \geq 2$. We analyze 
the pairwise dot 
product between 
their rows and relate the results obtained in the previous section to 
classify all matrices of this form.

\begin{definition}\label{def2}
    Call a $m \times n$ matrix $M$ \textbf{$m$-Hadamardesque} if for every column 
    $c \in M$, $c = pc'$ for some column $c' \in T_m$ and $p > 0$. 
\end{definition}
Note that we do not require every column of $T_m$ to have a multiple that shows up in $M$, and it is also acceptable for a column from $T_m$ to have multiples that appear more than once in $M$.

Now we want to classify $m$-Hadamardesque matrices.
%\begin{definition}\label{def1}
 %   Define the \textbf{Column Representation Vector} of a $m \times n$ $m$-Hadamardesque matrix $M$
  %  to be the vector $v \in \RR^{2^{m-1}}$ such that $v_i$ is equal to 
  %  $\sum_{j}p_{j}^2$ such that $c_i$ is the 
 %   $i^{th}$ column in $T_m$ and $p_jc_i$ is a column in $M$.
%\end{definition}
\begin{definition}\label{def1} Let $M$ be a $m \times n$ $m$-Hadamardesque matrix. By definition, the $j^{th}$ column of $M$ is of the form $p_jc_i$ for some $p_j>0$ and some column $c_i$ in $T_m$. 
    We define the \textbf{Column Representation Vector} of $M$
    to be the vector $v \in \RR^{2^{m-1}}$ whose the coordinate $v_i$ is equal to 
    $\sum_{j}p_{j}^2$ if $p_jc_i$ is a column in $M$ (and 0 otherwise).
\end{definition}
\begin{ex}
Let \[M = \begin{bmatrix}
2&1&\sqrt{3}&1&\sqrt{5}&3\\
2&-1&\sqrt{3}&-1&\sqrt{5}&-3\\
2&1&\sqrt{3}&-1&\sqrt{5}&3\\
2&-1&\sqrt{3}&-1&-\sqrt{5}&3\\
\end{bmatrix}
\]
and let $v$ be its Column Representation Vector.\\
\\
Since $M$ has 4 rows, its column vectors come from
\[T_4 = \begin{bmatrix}
1&1&1&1&1&1&1&1\\
1&-1&1&-1&1&-1&1&-1\\
1&1&-1&-1&1&1&-1&-1\\
1&1&1&1&-1&-1&-1&-1
\end{bmatrix}
\]
Let $c_i$ be the $i^{th}$ column of $T_4$.\\
\\
The first and third columns of $M$ are $2c_1$ and $\sqrt{3}c_1$ respectively. 
As such $v_1 = 4 + 3 = 7$.\\
The second column of $M$ is $c_6$. As such $v_6 = 1$.\\
The fourth column of $M$ is $c_8$. As such $v_8 = 1$.\\
The fifth column of $M$ is $\sqrt{5}c_5$. As such $v_5 = 5$.\\
The sixth column of $M$ is $3c_2$. As such $v_2 = 9$.\\
The columns $c_3,c_4,c_7$ do not appear in M and so $v_3=v_4=v_7=0$.
Hence $v = (7,9,0,0,5,1,0,1)$.
\end{ex}

\begin{proposition}\label{prop5}
    Let $1 \leq r_i < r_j \leq m$. Then the dot product of the $r_i^{th}$ and 
    $r_j^{th}$ rows in a m-Hadamardesque matrix is equal to the dot product of 
    its Column Representation Vector and the $(\frac{(r_j-1)(r_j-2)}{2} + 
    r_i)^{th}$ row in $CT_m$.
\end{proposition}
\begin{proof}
This follows directly from the definitions of $m$-Hadamardesque matrices and the Column Representation Vector.     
\end{proof}

\begin{theorem}\label{thm_ortho}
A m-Hadamardesque matrix has pairwise orthogonal rows iff its Column 
Representation Vector is in the span of $RC_m$.
\end{theorem}

\begin{proof}
Since $R_m$ is the set of rows of $CT_m$, and $RC_m$ spans the orthogonal 
complement of $R_m$, we see this follows immediately from Proposition 
\ref{prop5}.
\end{proof}

This result immediately classifies $m \times m$ Hadamard matrices in terms of $m$-Hadamardesque matrices.

\begin{theorem} Let $M$ be a $m \times m$ matrix. The following are equivalent. 
\begin{enumerate}
\item $M$ is a Hadamard matrix of order $m$.
\item $M$ is a (square) $m$-Hadamardesque matrix with entries $\pm 1$ and the Column 
Representation Vector $v$ is 
in the span of $RC_m$. 
\item $M$ is a (square) $m$-Hadamardesque matrix with the Column 
Representation Vector $v$ having $m$ ones and $2^{m-1} - m$ zeros, and $v$ is 
in the span of $RC_m$. 
\end{enumerate}
\end{theorem}
\begin{proof}
$(1\Rightarrow 2):$ This is immediate from Theorem \ref{thm_ortho}.

$(2\Rightarrow 3):$ Since $M$ has entries $\pm 1$ and is a square matrix of order $m$, there is exactly $m$ columns of $T_m$, each shows up one time in $M$. The conclusion follows.

$(3\Rightarrow 1):$ By definition, since the Representation Vector $v$ has $m$ ones and $2^{m-1} - m$ zeros, there are exactly $m$ columns of $T_m$ that appear in $M$. Moreover, since $M$ is a square matrix, these columns show up once in $M$. Theorem \ref{thm_ortho} guarantees that the rows of $M$ are orthogonal, so $M$ is a Hadamard matrix.
\end{proof}
\begin{remark}
There is non-square $m$-Hadamardesque matrix with the Column 
Representation Vector $v$ having $m$ ones and $2^{m-1} - m$ zeros, and $v$ is 
in the span of $RC_m$. For example, let 
\[M = \begin{bmatrix}
\frac{1}{\sqrt{2}}&\frac{1}{\sqrt{2}}&1&1&1\\
\frac{1}{\sqrt{2}}&\frac{1}{\sqrt{2}}&-1&1&-1\\
\frac{1}{\sqrt{2}}&\frac{1}{\sqrt{2}}&1&-1&-1\\
\frac{1}{\sqrt{2}}&\frac{1}{\sqrt{2}}&-1&-1&1\\
\end{bmatrix}
\]
Then $v=(1, 0, 0, 1, 0, 1, 1, 0)$.
\end{remark}

These results also have a number of other important corollaries.

\begin{corollary} Let $M$ be a $m \times m$ matrix with entries of the form $\pm 1$. Then $M$ is a partial Hadamard matrix if and only if $M$ is an m-Hadamardesque matrix with 
Column Representation Vector $v$ having $n$ ones and $2^{m-1} - n$ zeros, and $v$ is 
in the span of $RC_m$. 
\end{corollary}

\begin{corollary}
    Given two m-Hadamardesque matricies $M$ and $M'$ with Column Representation 
    Vector $v$ and $v'$ respectively. $M$ and $M'$ have the same pairwise row 
    dot product iff $v-v'$ is in the span of $RC_m$.
\end{corollary}

We will continue illustrate the power of our approach by proving a few more theorems.
\begin{theorem}\label{thm1}
    For every $a \in \mathbb{R}^{\frac{m(m-1)}{2}}$ there exists a $m$-Hadamardesque matrix $M$ such that $\langle r_i, r_j \rangle = 
    a_{\frac{(r_j-1)(r_j-2)}{2} + r_i}$ where $r_i$ and $r_j$ are the $i^{th}$ 
    and $j^{th}$ rows of $M$ and $j > i$. 
\end{theorem}
\begin{proof} If $a=0$, then we can let $M=T_m$. Now consider $a\not=0$. To construct the matrix $M$, it is enough to contruct its Column Representation 
    Vector $v$. Define
$$v'= \frac{1}{2^{m-1}}\sum_{i=1}^{\frac{m(m-1)}
    {2}}a_{i}R_{i},$$ where $R_i$ is the $i^{th}$ row of $CT_m$. Now let $v=v'+s\bf{1}$, where $\bf 1$ is the all 1's vector in $RC_m$ and $s$ is a number so that all coordinates of $v$ are non-negative. Let $M$ be a $m$-Hadamardesque matrix whose the Column Representation 
    Vector is $v$. Let $r_i, r_j$ be any two rows of $M$, $(i<j)$ and $R_{ij}$ be the $(\frac{(r_j-1)(r_j-2)}{2} + 
    r_i)^{th}$ row in $CT_m$. By Proposition \ref{prop5}, we have that 
\[\langle r_i, r_j \rangle = \langle v, R_{ij}\rangle = \langle v'+a{\bf 1}, R_{ij}\rangle=\langle v', R_{ij}\rangle = 
    s_{\frac{(r_j-1)(r_j-2)}{2} + r_i}.\] This comptetes the proof.
\end{proof}
\begin{remark} The matrix $M$ in Theorem \ref{thm1} is not unique since two different $m$-Hadamardesque matrices might have the same Column Representation 
    Vector.
\end{remark}

\begin{theorem}\label{thm2}
 If $a$ in Theorem \ref{thm1} is also in $\mathbb{Q}^{\frac{m(m-1)}{2}}$, then the m-Hadamardesque matrix $M$ can be taken to have entries only of the form $\pm r$ for some $r \in \mathbb{Q}$. 
\end{theorem} 
\begin{proof} We first show that if $a\in \mathbb{Q}^{\frac{m(m-1)}{2}}$, then the $m$-Hadamardesque matrix $M$ can be taken to have entries in $\mathbb{Q}$. Let $v'$ be the vector constructed as in the proof of Theorem \ref{thm1}. Adding any 
    rational multiple of the all $1's$ vector to $v'$ so that its coordinates 
    are non-negative would give a vector $v$ with non-negative rational coordinates. To contruct the matrix $M$, for each $v_i \neq 0$, if $v_i = \frac{p}{q}$, we may simply take $pq$ copies of $\frac{1}{q}c_i$ where $c_i$ is the $i^{th}$ column of $T_m$.

Now we will construct the matrix $M'$ with entries of the form $\pm r$ for some $r\in \mathbb{Q}$. Let $M$ be the matrix we have just constructed above. 
    Then each column $c_i$ in $M$ is of the form $(\frac{1}{d_i}, \pm\frac{1}{d_i}, \ldots,
    \pm\frac{1}{d_i}) \in \mathbb{Q}^{m}$. Let $d$ be the least common multiple 
    of all the $d_i$'s. For each column $c_i$ in $M$, we let  $c_i'$ to 
    be $c_i$ but replacing the denominator $d_i$ with $d$. We then replace $c_i$ with 
    $(\frac{d}{d_i})^2$ copies of $c'_i$. The resulting matrix $M'$ only contains entries of the form $\pm\frac{1}{d}$, and it has the same pairwise dot product as the original matrix $M$.
\end{proof}

\begin{corollary}\label{cor1}
     If $a$ in Theorem \ref{thm1} is in $\mathbb{Q}^{\frac{m(m-1)}{2}}$, 
     then the $m$-Hadamardesque matrix $M$ can be taken to have entries only of 
     the form $\pm r$ for some $r \in 
     \mathbb{R} \setminus \mathbb{Q}$.
\end{corollary}
\begin{proof}
    We simply take the matrix in Theorem \ref{thm2} and take two copies of each 
    column. Then we divide each column by $\sqrt{2}$.
\end{proof}

\begin{remark}  Theorem \ref{thm2} and Corollary \ref{cor1} are not true for some $a 
\notin \mathbb{Q}^{\frac{m(m-1)}{2}}$. For instance, let $a$ be the vector containing both non-zero rational and irrational coordinates. Let $M$ be any matrix with entries of the form $\pm r$ for some $r\in \mathbb{Q}$ (or $r\in \mathbb{R}\setminus \mathbb{Q})$. The conclusion folows by observing that the dot product of any two rows of $M$ is of the form $sr^2$ for some integer $s$.
\end{remark}

\end{document}